\documentclass[a4 paper, 11 pt]{article}
\usepackage{amsmath,amsthm,amssymb,amsfonts,amscd, graphicx}
\usepackage{amssymb, amsmath, amsthm}
\usepackage{graphicx}
\usepackage{tikz-cd}
\usepackage{caption}
\usepackage{subcaption}

\usepackage[all]{xy}
\usepackage[colorinlistoftodos]{todonotes}
\usepackage{float}
\usepackage{comment}
\usepackage{mathtools}
\usepackage{soul}

\usepackage{hyperref}
\usepackage[left=1 in,top=1.2 in, right=1 in, bottom=1.2 in]{geometry}
\newtheorem{theorem}{Theorem}[section]
\newtheorem{lemma}[theorem]{Lemma}
\newtheorem{proposition}[theorem]{Proposition}

\theoremstyle{definition}

\newtheorem{remark}[theorem]{Remark}

\newtheorem{problem}[theorem]{Problem}

\numberwithin{equation}{section}
\usepackage{lineno}
\def\BigOh{\mathop{\rm O}}

\title{\bfseries{The non-existence of some Moore polygons and spectral Moore bounds}}
\author{Sebastian Cioab\u{a}\thanks{Department of Mathematical Sciences, 
	University of Delaware, 
	Newark, DE 19716, 
	USA, {\tt cioaba@udel.edu}. S.M. Cioab\u{a} was partially supported by NSF grant DMS-2245556.}
\and
Vishal Gupta\thanks{Department of Mathematical Sciences, 
	University of Delaware, 
	Newark, DE 19716, 
	USA, {\tt vishal@udel.edu}.}
\and 
Hiroshi Nozaki\thanks{Department of Mathematics Education, 
	Aichi University of Education, 
	1 Hirosawa, Igaya-cho, 
	Kariya, Aichi 448-8542, 
	Japan, {\tt hnozaki@auecc.aichi-edu.ac.jp}. H. Nozaki was partially supported by JSPS KAKENHI Grant Numbers 22K03402 and 24K06688.}
    \and
Ziqing Xiang\thanks{Department of Mathematics and National Center for Applied Mathematics, 
Southern University of Science and Technology, 
1088 Xueyuan Avenue, Shenzhen 518055, 
China, {\tt xiangzq@sustech.edu.cn}. Z. Xiang was partially supported by NSFC grant 12350710787 and 12471311.}}
\date{\today}

\begin{document}
\maketitle

\begin{abstract}
In this paper, we study the maximum order $v(k, \theta)$ of a connected $k$-regular graph whose second largest eigenvalue is at most $\theta$. From Alon–Boppana and Serre, we know that $v(k, \theta)$ is finite when $\theta < 2\sqrt{k-1}$ while the work of Marcus, Spielman, and Srivastava implies that $v(k,\theta)$ is infinite if $\theta\geq 2\sqrt{k-1}$. Cioab\u{a}, Koolen, Nozaki, and Vermette obtained a general upper bound on $v(k, \theta)$ via Nozaki's linear programming bound and determined many values of $v(k,\theta)$. The graphs attaining this bound are distance-regular and are called Moore polygons. Damerell and Georgiacodis proved that there are no non-trivial Moore polygons of diameter $6$ or more. For smaller diameters, there are infinitely many Moore polygons.

We complement these results by proving two nonexistence results for Moore polygons with specific parameters. We also determine new values of $v(k,\theta)$: $v(4, \sqrt{2}) = 14$ and $v(5, \sqrt{2}) = v(5,\sqrt{5}-1)=16$. The former is achieved by the co-Heawood graph, and the latter by the folded $5$-cube. We verify that any connected $5$-regular graph with second eigenvalue $\lambda_2$ exceeding $1$ satisfies $\lambda_2 \geq \sqrt{5} - 1$, and that the unique $5$-regular graph attaining equality in this bound has $10$ vertices. We prove a stronger form of a 2015 conjecture of Kolokolnikov related to the second eigenvalue of cubic graphs of given order, and observe that other recent results on the second eigenvalue of regular graphs are consequences of the general upper bound theorem on $v(k,\theta)$ mentioned above.
\end{abstract}

\section{Introduction}

Let $k=\lambda_1>\lambda_2\geq \cdots \geq \lambda_n$ be the eigenvalues of the adjacency matrix of a connected $k$-regular graph with $n$ vertices. The difference $k-\lambda_2$ is usually called the \emph{spectral gap} of the graph. Regular graphs with a large spectral gap $k-\lambda_2$ are of great interest as they have good connectivity and expansion properties \cite{A1986,HLW2006,CKY2021}. 

In this paper, we study the maximum order of a $k$-regular connected graph with a given upper bound on the second eigenvalue. Let $v(k,\theta)$ denote the maximum order of a connected $k$-regular graph whose second largest eigenvalue is at most $\theta$. From Alon--Boppana and Serre \cite{A1986,HLW2006,Serre}, we know that $v(k,\theta)$ is finite if $\theta < 2\sqrt{k-1}$, while the work of Marcus, Spielman, and Srivastava \cite{MSS2015} implies that $v(k,\theta)$ is infinite if $\theta\geq 2\sqrt{k-1}$. In \cite{RSS}, Richey, Stover, and Shutty used a quantitative version of Serre's proof from \cite{Serre} to determine $v(k,\theta)$ for several parameters and made some conjectures related to $v(k,\theta)$. Cioab\u{a}, Koolen, Nozaki, and Vermette~\cite{Sebisiamvktheta} settled these conjectures and determined many values of $v(k,\theta)$. The same problem has been studied for bipartite regular graphs \cite{vkthetaBipartite} and regular uniform hypergraphs \cite{CKMNO2022}, see also \cite{CioabaRIMS,FJ2021,Lato, YK2021}.

To describe the results of this paper, we recall the following general result for $v(k, \theta)$ obtained in \cite{Sebisiamvktheta}, as it is the foundation of our analysis. Denote by $\lambda_2(T)$ the second largest eigenvalue of a matrix $T$ and by $\lambda_2(G)$ the second largest eigenvalue of the adjacency matrix of a (regular) graph $G$. Let $T(k,t,c)$ be the $t \times t$ tridiagonal matrix with constant row sum $k$, with lower diagonal $(1, 1, \ldots, 1, c)$ and upper diagonal $(k, k - 1, \ldots, k - 1)$, 
\[
T(k, t, c) =
\begin{bmatrix}
0     & k     &        &        &        &        \\
1     & 0     & k-1    &        &        &        \\
      & 1     & 0      & \ddots &        &        \\
      &       & \ddots & \ddots & k-1    &        \\
      &       &        & 1      & 0      & k-1    \\
      &       &        &        & c      & k-c
\end{bmatrix}
\]
for $k\geq 3$, $t\geq 3$, and $c>0$. 

Using Nozaki's linear programming bound for graphs \cite{NozakiLP}, Cioab\u{a}, Koolen, Nozaki, and Vermette  \cite{Sebisiamvktheta} proved that if $\theta=\lambda_2(T(k, t, c))$, then
\begin{equation} \label{eq:vklambda}
v(k, \theta) \leq M(k,t,c)= 1 + \sum_{i=0}^{t-3} k(k - 1)^i + \frac{k(k - 1)^{t-2}}{c},
\end{equation}
and characterized the equality case. A regular graph attaining the bound \eqref{eq:vklambda} must be distance-regular with intersection array $\{k, k - 1, \ldots, k - 1;\ 1, \ldots, 1, c\}$. 
Damerell and Georgiacodis \cite{damerellmoorepolygon} called 
such distance-regular \emph{Moore polygons}. 
A Moore polygon is {\it trivial} if $c=1$, $c=k$,
or $k\leq 2$.
They proved that no non-trivial Moore polygons exist
when the diameter is $6$ or more. 
However, for smaller diameters, there are infinitely many values of $k$ and $c$ for which Moore polygons exist, see \cite{Sebisiamvktheta}. The classification of Moore polygons remains open. 

In Section \ref{sec:nonex_Moorepolygons}, we contribute to this work and establish the following:
\begin{enumerate}
    \item There is no distance-regular graph with intersection array
    \[
    \{k, k - 1, k - 1;\ 1, 1, k - \sqrt{k}\}
    \]
    for $k\geq  3$, except for $k=4$. In that case, there is a unique distance-regular graph with intersection array $\{4,3,3;1,1,2\}$, namely the Odd graph $O_4$. Note that $\lambda_2(T(k,4,k-\sqrt{k}))=\sqrt{k}$. 
    
    \item There exists no distance-regular graph with intersection array
    \[
    \{k, k - 1, k - 1, k - 1;\ 1, 1, 1, k - \sqrt{2k - 1}\}
    \]
    for $k\geq 3$. Note that $\lambda_2(T(k,5,k-\sqrt{2k-1}))=\sqrt{2k-1}$.
\end{enumerate}

In Section \ref{sec:v5sqrt5min1}, we prove that $v(5, \sqrt{5} - 1) = 16$ and, together with the results of Section \ref{sec: jumps in the second eigenvalue}, this shows that the folded $5$-cube\footnote{This graph is obtained from the $5$-dimensional binary cube by identifying antipodal vertices and is the unique strongly regular graph with parameters $(16,5,0,2)$. It is called the Clebsch graph by some authors \cite[p. 226]{GodsilRoyle} while other authors \cite[p.117]{BrouwerHaemers} call it the complement of the Clebsch graph.} is the only $5$-regular graph on $16$ vertices with second largest eigenvalue at most $\sqrt{5}-1$. 

In Section \ref{sec: jumps in the second eigenvalue}, we show that if a connected $5$-regular graph has second eigenvalue $\lambda_2 > 1$, then 
$\lambda_2 \geq \sqrt{5} - 1$. We also prove that the only $5$-regular graph with $\lambda_2 = \sqrt{5} - 1$ is the one in Figure~\ref{fig: n=10}. 

In Section \ref{sec:vksqrt2}, we determine two new values of $v(k, \theta)$: 
\[
v(4, \sqrt{2}) = 14, \quad v(5, \sqrt{2}) = 16.
\]
The co-Heawood graph achieves $v(4, \sqrt{2}) = 14$ and has second eigenvalue $\sqrt{2}$. 
The folded $5$-cube achieves $v(5, \sqrt{2}) = 16$, although its second eigenvalue is $1$.

As a consequence, the tables of $v(k, \theta)$ for $k\in \{4, 5\}$ are updated as follows:

\begin{center}
\begin{tabular}{|c|ccccccccc|}
\hline
 $\theta$    &  $-1$&$0$&$1$&$\sqrt{5}-1$& $\sqrt{2}$& $\sqrt{3}$&$2$&$\sqrt{6}$&3  \\ 
 \hline
 $v(4,\theta)$    & 5&8&12&12&14&26&35&80&728\\
 \hline
\end{tabular}
\end{center}

\begin{center}
\begin{tabular}{|c|cccccccc|}
\hline
 $\theta$    &  $-1$&$0$&$1$&$\sqrt{5}-1$& $\sqrt{2}$&2& $2\sqrt{2}$&$2\sqrt{3}$  \\ 
 \hline
 $v(5,\theta)$    & 6&10&16&16&16&42&170&2730\\
 \hline
\end{tabular} 
\end{center}
We point out that Table 1 in \cite{Sebisiamvktheta} contains typos in $v(4,1)$ and $v(4,\sqrt{5}-1)$ and that Table 2 in \cite{CKMNO2022} provides additional information on $v(k,\theta)$. 

In Section \ref{sec:7}, we obtain a linear programming bound on the order of a $k$-regular graph with given girth, leading to lower bounds on its second eigenvalue. Using this, we reprove the Moore bound and refine the algebraic-connectivity bound given by Exoo, Kolokolnikov, Janssen, and Salamon \cite{EKSJ2023} for known $(k,g)$-cages.

In Section \ref{sec:6}, we use \eqref{eq:vklambda} to prove a conjecture of Kolokolnikov \cite{Kolo15}, and to give an alternative proof of some lower bounds on $\lambda_2$ given in \cite{Balla}. 

\section{The non-existence of some Moore polygons}\label{sec:nonex_Moorepolygons}

The following problem is stated in \cite{Sebisiamvktheta}. 
\begin{problem}[Problem 5.1 in \cite{Sebisiamvktheta}]
    Determine $v(k,\sqrt{k})$ for $k\geq 3$. 
\end{problem}

For $t=4$ and $c=k-\sqrt{k}$, we have that 
\begin{equation}
T(k, 4, k-\sqrt{k}) = \begin{bmatrix}
    0&k&0&0\\
    1&0&k-1&0\\
    0&1&0&k-1\\
    0&0&k-\sqrt{k}&\sqrt{k}
\end{bmatrix}.
\end{equation}
and $\lambda_2(T(k,4,k-\sqrt{k})) = \sqrt{k}$. Using \eqref{eq:vklambda}, we obtain that
\begin{equation}\label{sqrt k}
v(k,\sqrt{k})\leq 2k^2 + k^{3/2}-k-\sqrt{k}+1.
\end{equation}
If a graph attains this bound, then it is a distance-regular graph with the intersection array $\{b_0,b_1,b_2;c_1,c_2,c_3\}=\{k,k-1,k-1;1,1,k-\sqrt{k}\}$. For small $k$, we know the following:
\begin{itemize}
        \item $v(2, \sqrt{2}) = 8$ attained by the cycle $C_8$.
        \item $v(3, \sqrt{3}) = 18$ attained by the Pappus graph.
        \item $v(4, 2) = 35$ attained by the Odd graph $O_4$.
\end{itemize}

If $q$ is a prime power, let $IG(AG(2,q)\setminus pc)$ be the incidence graph of the finite affine plane $AG(2, q)$ minus a parallel class of lines. This is sometimes called a \emph{bi-affine plane}.

The graph $IG(AG(2,5)\setminus pc)$ is a distance-regular graph on $50$ vertices with spectrum $\{5^1, \sqrt{5}^{20}, 0^8, -\sqrt{5}^{20}, -5^1\}$. Therefore, using \eqref{sqrt k}, we get that $50\leq v(5,\sqrt{5})\leq 54$. 

The graph $IG(AG(2, 7)\setminus pc)$ is a distance-regular graph on 98 vertices, and its spectrum is
$\{7^{(1)}, \sqrt{7}^{(42)}, 0^{(12)}, -\sqrt{7}^{(42)}, -7^{(1)}\}$. Using \eqref{sqrt k}, we have that $98\leq v(7,\sqrt{7})\leq 106$. 

The graph $IG(AG(2, 9)\setminus pc)$ is a distance-regular graph on 162 vertices, and its spectrum is
$\{9^{(1)}, 3^{(72)}, 0^{(16)}, -3^{(72)}, -9^{(1)}\}$. Using \eqref{sqrt k}, we have that $162\leq v(9, 3)\leq 178$. 


The cycle graph $C_8$ is $IG(AG(2, 2)\setminus pc)$, and the Pappus graph is $IG(AG(2, 3)\setminus pc)$. Is there a pattern? Certainly, $v(4, 2)=35$ breaks the pattern, but at least we can say that when $k$ is a prime power, we have that
\begin{equation}
2k^2\leq v(k,\sqrt{k})\leq 2k^2 + k^{3/2}-k-\sqrt{k}+1,
\end{equation}
see also Table 1 in \cite{vkthetaBipartite}.

The degree $k$ should be square since the parameter $c_3$ is an integer. The only known graph with these parameters is the Odd graph $O_4$ when $k=4$. 

We will show that for $k\neq 4$, there does not exist such a graph. 

\begin{theorem}\label{thm: nomoorepolygon1}
For $k\geq 3, k\neq 4$, there is no distance-regular graph with the intersection array $\{k,k-1,k-1; 1,1,k-\sqrt{k}\}$. 
\end{theorem}
\begin{proof} If $k=3$, then the intersection array $\{k,k-1,k-1; 1,1,k-\sqrt{k}\}$ contains non-integers and therefore, there is no distance-regular graph with this intersection array.

Let $k>4$. By contradiction, assume that a distance-regular graph $G$ exists with these parameters. 

Let $(F_i^k)_{i\geq 0}$ be the sequence of orthogonal polynomials defined by the three-term recurrence relation:
\begin{equation} \label{eq:F012}
    F_0^k(x) = 1, F_1^k(x) = x, F_2^k(x) = x^2-k,
\end{equation}
and for $i\geq 3$,
\begin{equation} \label{eq:Fk}
    F_i^k(x) = xF_{i-1}^k(x) - (k-1)F_{i-2}^k(x).
\end{equation}
For simplicity we write $F_i(x)$ to denote $F_i^k(x)$. In \cite{Sebisiamvktheta}, the authors showed that the eigenvalues of the tridiagonal matrix $T(k,t,c)$ that are not equal to $k$ coincide with the roots of the polynomial $\sum_{i=0}^{t-2}F_i(x) + F_{t-1}(x)/c.$ 
Therefore, the non-trivial eigenvalues (eigenvalues that are not equal to $k$) of the distance-regular graph $G$ are the roots of the polynomial
\begin{equation}
  F(x)=(k-\sqrt{k})(F_0(x)+F_1(x)+F_2(x))+F_3(x)=(k-\sqrt{k})(1+x+x^2-k)+x^3-(2k-1)x.  
\end{equation}
In \cite{bannaiitospectra}, Bannai and Ito showed that the eigenvalues of a non-trivial Moore polygon are rational. Therefore, by the Rational Root Theorem, the eigenvalues of the distance-regular graph $G$ are all integers. 

By showing that the roots of the polynomial $F(x)$ are not integers, we establish the non-existence of $G$.

Let $s := \sqrt{k}$. The polynomial $F(x)$ can be rewritten as
\[
	F(x) = (x - s) (x^2 + s^2 x + s^3 - s^2 - s + 1).
\]
The discriminant of the quadratic factor is
\[
	D = s^4 - 4 s^3 + 4 s^2 + 4 s - 4.
\]
As $F(x)$ has only integer roots, the discriminant $D$ must be a perfect square. If $t := \sqrt{D}$, then
\begin{equation}
    s^4 - 4 s^3 + 4 s^2 + 4 s - 4 = t^2.
\end{equation}
Note that
\begin{equation}
    (s^2 - 2s - 1)^2 \leq s^4 - 4 s^3 + 4 s^2 + 4 s - 4 \leq (s^2 - 2s + 1)^2
\end{equation}
for integers $s$ except $-1, 0, 1, 2, 3$. Therefore, either $-1 \leq s \leq 3$, or $t = s^2 - 2s$. Moreover, the latter forces $s = 1$. A direct calculation with $-1 \leq s \leq 3$ shows that the set of integral points $(s, t)$ for $D = t^2$ is $$\{(-1, \pm 1), (1, \pm 1), (2, \pm 2)\}. $$ 
Therefore, $k = 1$ or $k = 4$, and the assertion holds. 
\end{proof}
When $k = 1$, the roots of $F(x)$ are $\{-1, 0, 1\}$, and the corresponding graph does not exist. 
When $k = 4$, the roots of $F(x)$ are $\{-3, -1, 2\}$, and the resulting graph is the Odd graph $O_4$. 

Note that in the above case, the second-largest eigenvalue $\sqrt{k}$ is the largest zero of $F_2(x)$. 
We consider the case where the second-largest eigenvalue is the largest zero of $F_i(x)$ for $i\geq 3$. 
For $i=3$, the largest zero of $F_3(x)$ is $\sqrt{2k-1}$. 
From \eqref{eq:vklambda} we have that
\[
v(k,\sqrt{2k-1}) \leq 1+k+k(k-1)+k (k-1)^2+k (k-1)\left(k+\sqrt{2 k-1}\right).  
\]
If a graph $G$ attains this bound, then $G$ is a distance-regular graph with the intersection numbers
\begin{equation}
    \{b_0,b_1,b_2,b_3;c_1,c_2,c_3,c_4\}=\{k,k-1,k-1,k-1;1,1,1,k-\sqrt{2k-1}\}.
\end{equation}
The non-trivial eigenvalues of $G$ coincide with the zeros of the polynomial
\begin{align}
F(x)&=c_4(F_0(x)+F_1(x)+F_2(x)+F_3(x))+F_4(x) \nonumber\\
&= (x-\sqrt{2k-1})\left(x^3+k x^2+(k-1) (\sqrt{2 k-1}-2) x-k+1 \right) \nonumber\\
&=\frac{1}{2}(x-s)\left(2 x^3+(s^2+1) x^2+(s^2-1)(s-2)x-s^2+1 \right), \label{eq:ff}
\end{align}
where $s=\sqrt{2k-1}$. 
Using \cite{bannaiitospectra} and the Rational Root Theorem, we have that the eigenvalues of $G$ are all integers. 
\begin{theorem}\label{thm: nomoorepolygon2}
For $k\geq 3$, there is no distance-regular graph with the intersection array $(k,k-1,k-1,k-1; 1,1,1, k-\sqrt{2k-1})$. 
\end{theorem}
\begin{proof}
Let $k\geq 3$. By contradiction, assume that a distance-regular graph $G$ exists with these parameters. Let $s=\sqrt{2k-1}\in \mathbf{N}$. We prove that the polynomial  
    \[
    f(x)=x^3+kx^2+(k-1)(s-2)x-k+1,
    \]
which comes from a factor of $F(x)$ in \eqref{eq:ff}, has a non-integer root.

Let $\alpha_1,\alpha_2,\alpha_3$ be the roots of $f(x)$. 
Assume all $\alpha_i$ are integers and $|\alpha_1|\leq |\alpha_2|\leq |\alpha_3|$. 
The roots $\alpha_1,\alpha_2,\alpha_3$ are all distinct (see \cite[Proof of Theorem~2.3]{Sebisiamvktheta}). 
Therefore, we have $|\alpha_3| \geq |\alpha_1|+1$. 
By Vieta's formula, we have 
\begin{equation} \label{eq:roots_and_coef}
    \alpha_1 \alpha_2 \alpha_3=k-1 \ne 0, \qquad 
    \alpha_1+\alpha_2+\alpha_3=-k.
\end{equation}
    Then, we have
\begin{equation} \label{eq:k-1}
    k-1=|\alpha_1|\cdot | \alpha_2|\cdot |
\alpha_3|\geq |\alpha_1|^2\cdot(|\alpha_1|+1) >|\alpha_1|^3. 
\end{equation}
Moreover, it follows that
\begin{equation} \label{eq:mod}
0=f(\alpha_1)\equiv \alpha_1^3+\alpha_1^2 \pmod{k-1}.    
\end{equation}

If  $\alpha_1<0$, then    
\[
|\alpha_1^3+\alpha_1^2|=|\alpha_1|^2 \cdot |\alpha_1+1|<|\alpha_1|^3 <k-1. 
\]
 This implies $\alpha_1^3+\alpha_1^2=0$ from \eqref{eq:mod}, and $\alpha_1=-1$. 
 From \eqref{eq:roots_and_coef}, the other two roots are 
 \[
 \alpha_2 = \frac{-(k-1) - \sqrt{(k-1)(k+3)}}{2}, \alpha_3 = \frac{-(k-1) + \sqrt{(k-1)(k+3)}}{2}. 
 \]
 Since they are integers, there exists $t\in \mathbb{N}$ such that $t^2=(k-1)(k+3)$. We rewrite this equality as
 \begin{equation}
     (k+1+t)(k+1-t)=4.
 \end{equation}
 This equality has no integer solution $(k,t)\in \mathbb{N} \times \mathbb{N}$. Therefore,  $\alpha_2,\alpha_3$ are not integers. 

 Suppose $\alpha_1>0$ holds. Then $|\alpha_2| \geq \alpha_1$ and $|\alpha_3| \geq \alpha_1+1$. It follows that 
 \[
0<\alpha_1^3+\alpha_1^2=\alpha_1^2(\alpha_1+1)\leq \alpha_1 \cdot |\alpha_2| \cdot |\alpha_3|=k-1. 
 \]
 This implies $\alpha_1^3+\alpha_1^2=k-1$ from \eqref{eq:mod}, and hence $|\alpha_2|=\alpha_1$ and $|\alpha_3|=\alpha_1+1$. From 
 $\alpha_1\alpha_2\alpha_3=k-1>0$, we have $\alpha_2<0, \alpha_3<0$. Therefore, $\alpha_2=-\alpha_1$ and $\alpha_3=-(\alpha_1+1)$. From \eqref{eq:roots_and_coef}, the root $\alpha_1$ satisfies $\alpha_1(\alpha_1^2+\alpha_1-1)=0$, which has no positive integer solution. 

Therefore, $f(x)$ has non-integer roots, which contradicts the existence of the distance-regular graphs.  
\end{proof}

Using Theorem \ref{thm: nomoorepolygon2} and the result of Damerell and Georgiacodis \cite{damerellmoorepolygon}, we conclude that there are no non-trivial $k$-regular Moore polygons with second largest eigenvalue $\sqrt{2k-1}$. It is easy to see there is no tridiagonal matrix $T(k, t, c)$ with $\lambda_2(T(k,t,c)) = \sqrt{2k-1}$ for $t\leq 3$. For $t=4$, the only possibility is when $c=0$. The eigenvalues of $T(k, 4, 0)$ are $k, \sqrt{2k-1},0$, and $-\sqrt{2k-1}$. However, $c$ cannot be $0$ as the corresponding distance-regular graph must be connected. The case $t=5$ is done by Theorem \ref{thm: nomoorepolygon2}. For $t=6,$ the second largest eigenvalue of $T(k,t,c)$ is $\sqrt{2k-1}$ for 
$$c=\frac{(k-1)^2\sqrt{2k-1}}{-k^2+3k-1+\sqrt{2k-1}}<0$$
for $k>3$, and $c$ is irrational for $k=3$. 
Since Damerell and Georgiacodis \cite{damerellmoorepolygon} proved the non-existence of non-trivial Moore polygons for diameter $d>5$, we can stop at $t=6$.
 
\section{$v(5,\sqrt{5}-1)=16$}\label{sec:v5sqrt5min1}

In this section, we prove $v(5,\sqrt{5}-1)=16$, which is achieved by folded $5$-cube with second eigenvalue $1$.  In Section \ref{sec: jumps in the second eigenvalue}, we also show that the graph shown in Figure \ref{fig: n=10} of order $10$ is the unique $5$-regular graph with second eigenvalue $\sqrt{5}-1$.  

We will use the following result that is due to Tutte.
\begin{lemma}[\cite{Tuttegirth}]\label{girth restriction}
 Let $n(k,g)$ denote the minimum possible number of vertices of a $k$-regular graph with girth (the length of its shortest cycle) $g$. 
 \begin{enumerate}
\item[$(1)$] If $g=2d+1$ is odd, then
\[
n(k,g)\;\ge\; \frac{k(k-1)^{d}-2}{k-2}
\;=\;
1+k\sum_{i=0}^{d-1}(k-1)^i .
\]
\item[$(2)$] If $g=2d$ is even, then
\[
n(k,g)\;\ge\; \frac{2(k-1)^{d}-2}{k-2}
\;=\;
2\sum_{i=0}^{d-1}(k-1)^i .
\]
\end{enumerate}
\end{lemma}
\begin{proposition}\label{v(5,sqrt5 -1)}
  $v(5, \sqrt{5}-1) = 16.$  
\end{proposition}
\begin{proof}
Let $G$ be a $5$-regular graph with $\lambda_2(G)\leq \sqrt{5}-1$ on $n = v(5, \sqrt{5}-1)$ vertices. Consider the following matrix
$$T = \begin{bmatrix}
    0& 5 &0\\
    1&0& 4\\
    0&2-\frac{1}{\sqrt{5}}&3+\frac{1}{\sqrt{5}}
\end{bmatrix}.$$ Its second largest eigenvalue  $\lambda_2(T)= \sqrt{5}-1$.
By \eqref{eq:vklambda}, we get that $v(5,\sqrt{5}-1)\leq M(5,3,2-\frac{1}{\sqrt{5}})=18.8801.$ So $v(5,\sqrt{5}-1)\leq 18$. We know that $v(5,1) =16$ attained only by the folded $5$-cube \cite{Sebisiamvktheta}. Therefore, $16\leq v(5,\sqrt{5}-1)\leq 18$.  

Using Lemma \ref{girth restriction}, we observe that a 5-regular graph with girth more than 4 has at least 26 vertices, so $G$ has girth either 3 or 4. Suppose $G$ has girth 3. Let $H$ be a subgraph of $G$ isomorphic to $C_3$ and let $v\in G$ be not adjacent to any vertex of $H$. Consider the partition of $G$ into the following three parts : $V(H), \{v\}, V(G) \setminus (V(H)\cup\{v\}).$ The corresponding quotient matrix for a given order $n$ is 
$$Q_n = \begin{bmatrix}
  2&0&3\\
  0&0&5\\
  \frac{9}{n-4}&\frac{5}{n-4}&\frac{5n-34}{n-4}
\end{bmatrix}.$$
The second largest eigenvalue of this matrix is $\lambda_2(Q_n) = \frac{n-11+\sqrt{n^2-12n+81}}{n-4} > \sqrt{5}-1$ for $n\geq 12$. Since we can always find such a vertex $v\in G$ for $n\geq 14$, this proves that there is no connected 5-regular graph of order $n\geq 14$ that has girth 3 and second largest eigenvalue $\lambda_2\leq \sqrt{5}-1$. To complete the proof, we only need to show that there exists no 5-regular graph on 18 vertices with girth 4 and $\lambda_2 \leq \sqrt{5}-1$. 
 Suppose $G$ has girth 4 and $n=18$. Let $H$ be a subgraph of $G$ isomorphic to $C_4$ and let $v\in G$ be not adjacent to any vertex of $H$. Consider the partition of $G$ in to the following three parts : $V(H), \{v\}, V(G) \setminus (V(H)\cup\{v\}).$ The corresponding quotient matrix is
$$Q_G = \begin{bmatrix}
    2&0&3\\
    0&0&5\\
    \frac{12}{13}&\frac{5}{13}&\frac{48}{13}
\end{bmatrix}.$$ 
The second largest eigenvalue of $Q_G$ is $\lambda_2(Q_G)  = \frac{1}{26}(9 + \sqrt{601})>\sqrt{5}-1$. By eigenvalue interlacing, we get $\lambda_2(G) > \sqrt{5}-1,$ a contradiction. 
This completes the proof.
\end{proof}
\begin{figure}[htbp!]
    \centering
\begin{tikzpicture}[scale=1]
\foreach \i in {0,...,9} {
        \coordinate (v\i) at ({\i*36}:2); 
        \coordinate (l\i) at ({\i*36}:2.3);
    }
\foreach \i in {0,...,9} {
        \draw[fill] (v\i) circle (0.05cm);
        \node at (l\i) {\i};
    }
\foreach \i in {0,...,9} {
    \pgfmathtruncatemacro{\j}{mod(\i+1,10)}
    \draw (v\i) -- (v\j);
}
\foreach \i in {0,...,9} {
    \pgfmathtruncatemacro{\j}{mod(\i+2,10)}
    \draw (v\i) -- (v\j);
}
\foreach \i in {0,...,9} {
    \pgfmathtruncatemacro{\j}{mod(\i+5,10)}
    \draw (v\i) -- (v\j);
}
\end{tikzpicture}
    \caption{The Cayley graph $\mathrm{Cay}(\mathbf{Z}_{10},\{1,2,5,8,9\})$.}
   \label{fig: n=10}
\end{figure}
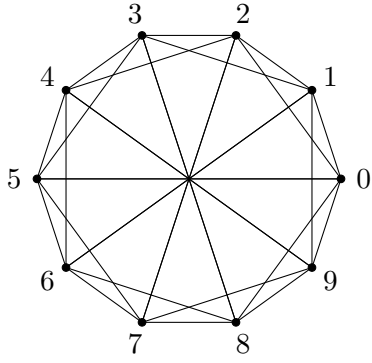

\section{Jumps in the second largest eigenvalue}\label{sec: jumps in the second eigenvalue}
Let $\rho$ be the unique real root of the cubic equation $x^3-x-1=0$ and let $\lambda^* = \rho^{1/2} + \rho^{-1/2}\approx 2.01980$.
\begin{proposition}
    Let $G$ be a $k$-regular graph of order $n$. If $\lambda_2(G)>1$, then $\lambda_2(G)> \lambda^*-1$.
\end{proposition}
\begin{proof}

In \cite{acharyajiang}, the authors showed that if a connected graph has its smallest eigenvalue in $(-\lambda^*, -2)$, then either it is isomorphic to an augmented path extension of a rooted graph (see Definition 1.2 in \cite{acharyajiang}) or it is isomorphic to one of the 4752  Maverick graphs provided as enum\_maverick\_graphs.txt file in the same paper. Let $G$ be a connected $k$-regular graph with $\lambda_2(G)\in (1, \lambda^*-1)$. Its complement $\overline{G}$ has $\lambda_{min}(\overline{G})\in (-\lambda^*, -2)$. We note that an augmented path extension of a rooted graph is an irregular graph, and a code\footnote{The code is available at https://github.com/vishalguptaud/Spectral-Moore-Problem/tree/main.} that computes the degree sequences of the Maverick graphs confirms that all of them are irregular as well. Hence, $\overline{G}$ is isomorphic to an irregular graph, a contradiction. To see the strict inequality in the result, we note that $-\lambda^*$ is not totally real and therefore cannot be an eigenvalue of a graph. This completes the proof.
\end{proof}

In \cite{Sebisiamvktheta}, the authors showed that if $G$ is a $4$-regular graph with $\lambda_2(G)>1$, then $\lambda_2(G)\geq \sqrt{5}-1$. We extend their result for $5$-regular graphs and also characterize the unique graph that attains the bound. We will use the following result.
\begin{theorem}[\cite{Koledin}] \label{lambda2 bound}
    Let $G$ be a connected $k$-regular bipartite graph with $n$ vertices and diameter $3$. Then 
    $$n\leq 2\frac{k^2-\lambda^2_2(G)}{k-\lambda^2_2(G)},$$ whenever the right-hand side is positive.
\end{theorem}

\begin{proposition}\label{gap in lambda2}
If $G$ is a connected 5-regular graph of order $n$ with $\lambda_2(G)>1,$ then $\lambda_2(G)\geq \sqrt{5}-1,$ with equality attained only by the Cayley graph of $(\mathbf{Z}_{10},+)$ with the generating set $\{1,2,5,8,9\}$ as shown in Figure \ref{fig: n=10}. 
\end{proposition}
\begin{proof}
Suppose $G$ is a $5$-regular graph with $1< \lambda_2(G)\leq \sqrt{5}-1$ on $n$ vertices. By Proposition~\ref{v(5,sqrt5 -1)}, we get $n\leq 16.$ Suppose $n=8$. We note that the complement of a $5$-regular graph on $8$ vertices is a $2$-regular graph on $8$ vertices. Therefore, $\overline{G}$ must be one of the following graphs $C_3\cup C_5$, $C_4 \cup C_4$, or $C_8$. The smallest eigenvalue of these graphs is $\frac{1}{2}(-1-\sqrt{5}), -2$, and $-2$, respectively. Consequently, the second largest eigenvalue of their complements can only be $\frac{1}{2}(\sqrt{5}-1)$ or $1$. This proves that $n$ cannot be $8$.

Using Lemma \ref{girth restriction}, we observe that a 5-regular graph with girth more than 4 has at least 26 vertices, so $G$ has girth either 3 or 4. Suppose $G$ has girth 3. Let $H$ be a subgraph of $G$ isomorphic to $C_3$ and let $v\in G$ be non-adjacent to any vertex of $H$. Consider the partition of $G$ in to the following three parts : $V(H), \{v\}, V(G) \setminus (V(H)\cup\{v\}).$ The corresponding quotient matrix for given order $n$ is 
$$Q_n = \begin{bmatrix}
  2&0&3\\
  0&0&5\\
  \frac{9}{n-4}&\frac{5}{n-4}&\frac{5n-34}{n-4}
\end{bmatrix}.$$
The second largest eigenvalue $\lambda_2(Q_n) = \frac{n-11+\sqrt{n^2-12n+81}}{n-4} > \sqrt{5}-1$ for $n\geq 12.$ Since we can always find such a vertex $v\in G$ for $n\geq 14,$ this proves that there is no connected 5-regular graph $G$ of order $n\geq 14$ that has girth 3 and $\lambda_2(G)\leq \sqrt{5}-1$. Suppose $n=12$ and $G$ does not have such a vertex $v$. This implies that for any subgraph $H\cong C_3$ of $G$ with $V(H) = \{a,b,c\}$, the set of remaining three neighbors of the vertices $a,b,c$ are mutually disjoint. Let $N'(i) = \{i_1, i_2, i_3\}$ be the set of remaining three neighbors of $i$ in $G$, where $i\in\{a, b, c\}.$  
\begin{itemize}
    \item[] Case 1: Suppose for some $i\in V(H)$, the induced subgraph $G[N'(i)]$ has two or more edges. Let $i = a$ and $a_1\sim a_2$ and $a_2\sim a_3$. Then $G[\{a, a_1, a_2\}]\cong C_3$ and $a_3$ is a common neighbor of $a$ and $a_2$, which is a contradiction. 
    \item[] Case 2: Suppose for some $i\in V(H)$, $G[N'(i)]$ has exactly one edge. Let $i = a$ and $a_1\sim a_2$. This means for $i\in \{b,c\}$, $G[N'(i)]$ must have at most one edge. Subcase 1: Suppose $a_1$ is adjacent to all the vertices in $N'(b)$. This means $a_2$ must be adjacent to all the vertices of $N'(c)$ because vertices $a, a_1, a_2$ induce a subgraph isomorphic $H' \cong C_3$ in $G$ and same as above, the vertices $a, a_1, a_2$ must not have any common neighbors in $G\setminus H'$. Note that $G[N'(b)], G[N'(c)]$ have no edges as otherwise, say if $b_1\sim b_2$, then we get a $G[\{b, b_1, b_2\}]\cong C_3$ and $a_1$ is a common neighbor of $b_1$ and $b_2$,  a contradiction. Suppose $a_3$ is adjacent to $c_1$ and all the vertices of $N'(b)$. The vertex $c_1$ must be adjacent to two vertices in $N'(b)$. Let $c_1\sim b_1$ and $c_1\sim b_2$. Then $G[\{a_3, b_1, c_1\}]\cong C_3$ and $b_2$ is a common neighbor of $a_3, c_1$, a contradiction. Next, suppose $a_3$ is adjacent to $b_1, b_2, c_1, c_2$. Then $b_3\sim c_i$ and $c_3\sim b_i$ for $i\in \{1,2,3\}$. The only subgraph $G\setminus H$ that meets the given restrictions is shown in Figure \ref{case 2 subcase 1 subgraph}. In that case, the second largest eigenvalue $\lambda_2(G)=1$.  
    \begin{figure}[h]
        \centering
\begin{tikzpicture}[scale=1]

\foreach \name/\x/\y in {
    v1/0/2, v2/0/1, v3/0/0, v4/2/0, v5/2/1, v6/2/2, v7/1/3, v8/-1.5/-0.6, v9/3.5/-0.6}
    {
        \coordinate (\name) at (\x,\y);
        \draw[fill] (\name) circle[radius=0.05cm];
    }
\foreach \u/\v in {
    v4/v1, v4/v2, v4/v3, v7/v1, v7/v2, v7/v5, v7/v6, v6/v2, v6/v3, v5/v1, v5/v3, v8/v1, v8/v2, v8/v3, v9/v4, v9/v5, v9/v6, v8/v9}
    {
        \draw (\u) -- (\v);
    }
    
\node[anchor=east] at (v1) {$b_1$};
\node[anchor=east] at (v2) {$b_2$};
\node[anchor=east] at (v3) {$b_3$};
\node[anchor=west] at (v4) {$c_3$};
\node[anchor=west] at (v5) {$c_2$};
\node[anchor=west] at (v6) {$c_1$};
\node[anchor=south] at (v7) {$a_3$};
\node[anchor=east] at (v8) {$a_1$};
\node[anchor=west] at (v9) {$a_2$};
\end{tikzpicture}
\caption{Subgraph $G\setminus H$ for Subcase 1 in Case 2.}
\label{case 2 subcase 1 subgraph}
\end{figure}
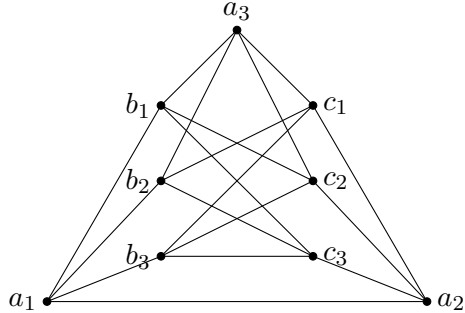
Subcase 2: Suppose $a_1$ is adjacent to $b_1, b_2, c_1$ and $a_2$ is adjacent to $b_3, c_2, c_3$. We first observe that $b_1\not\sim b_2$ and $c_2\not\sim c_3$.  If they were, it would result in a cycle $C_3$ with two endpoints sharing a common neighbor, which would lead to a contradiction. Suppose $a_3\sim b_1, b_2, b_3$. Depending on whether $a_3\sim c_1$ or $a_3\sim c_2$, the only possibilities (up to isomorphism of the resulting graphs) that satisfy the given conditions are $I_1, I_2$ as shown in Figure \ref{case 2 subcase 2 subgraph}. Next, if $a_3\sim b_1, b_2$, then based on the neighbors of $a_3$ in $N'(c)$, the possible induced subgraphs (up to isomorphism) are $I_3, I_4$ as shown in Figure \ref{case 2 subcase 2 subgraph}. Finally, if $a_3\sim b_2, b_3, c_1, c_2$, the possible induced subgraphs (up to isomorphism) that meet the given restrictions are $I_5, I_6$, and $I_7$ as shown in Figure \ref{case 2 subcase 2 subgraph}. We calculate the second largest eigenvalue for each graph in this subcase that contains one of these seven induced subgraphs. Notably, this eigenvalue equals $1$ for all seven graphs.
 \begin{figure}[h]
        \centering
\begin{tikzpicture}[scale=1]
\foreach \name/\x/\y in {
    v1/0/2, v2/0/1, v3/0/0, v4/2/0, v5/2/1, v6/2/2, v7/1/3
    }
    {
    \coordinate (\name) at (\x,\y);
    \draw[fill] (\name) circle[radius=0.05cm];
    }\foreach \u/\v in {
     v7/v1, v7/v2, v7/v3, v7/v6, v5/v6, v5/v1, v5/v2, v6/v3, v4/v1, v4/v2, v4/v3}
    {
        \draw (\u) -- (\v);
    }
\node[anchor=east] at (v1) {$b_1$};    
\node[anchor=east] at (v2) {$b_2$};
\node[anchor=east] at (v3) {$b_3$};
\node[anchor=west] at (v4) {$c_3$};
\node[anchor=west] at (v5) {$c_2$};
\node[anchor=west] at (v6) {$c_1$};
\node[anchor=south] at (v7) {$a_3$};
\node[anchor=north] at (1,-0.3) {$I_1$};
\end{tikzpicture}
\hspace{0.5cm}
\begin{tikzpicture}[scale=1]
\foreach \name/\x/\y in {
    v1/0/2, v2/0/1, v3/0/0, v4/2/0, v5/2/1, v6/2/2, v7/1/3
    }
    {
        \coordinate (\name) at (\x,\y);
        \draw[fill] (\name) circle[radius=0.05cm];
    }
\foreach \u/\v in {
     v7/v1, v7/v2, v7/v3, v7/v5, v5/v6, v4/v1, v4/v2, v4/v3, v5/v1, v6/v2, v6/v3}
    {
        \draw (\u) -- (\v);
    }
\node[anchor=east] at (v1) {$b_1$};
\node[anchor=east] at (v2) {$b_2$};
\node[anchor=east] at (v3) {$b_3$};
\node[anchor=west] at (v4) {$c_3$};
\node[anchor=west] at (v5) {$c_2$};
\node[anchor=west] at (v6) {$c_1$};
\node[anchor=south] at (v7) {$a_3$};
\node[anchor=north] at (1,-0.3) {$I_2$};
\end{tikzpicture}
\hspace{0.5cm}
\begin{tikzpicture}[scale=1]
\foreach \name/\x/\y in {
    v1/0/2, v2/0/1, v3/0/0, v4/1.5/0, v5/2/1, v6/2/2, v7/1/3
    }
    {
        \coordinate (\name) at (\x,\y);
        \draw[fill] (\name) circle[radius=0.05cm];
    }
\foreach \u/\v in {
     v7/v1, v7/v2, v7/v6, v7/v5, v6/v4, v4/v1, v4/v2, v2/v3, v5/v1, v5/v3, v3/v6}
    {
        \draw (\u) -- (\v);
    }    
\node[anchor=east] at (v1) {$b_1$};
\node[anchor=east] at (v2) {$b_2$};
\node[anchor=east] at (v3) {$b_3$};
\node[anchor=west] at (v4) {$c_3$};
\node[anchor=west] at (v5) {$c_2$};
\node[anchor=west] at (v6) {$c_1$};
\node[anchor=south] at (v7) {$a_3$};
\node[anchor=north] at (1,-0.3) {$I_3$};
\end{tikzpicture}
\hspace{0.5cm}
\begin{tikzpicture}[scale=1]
\foreach \name/\x/\y in {
    v1/0/2, v2/0/1, v3/0/0, v4/1.5/0, v5/2/1, v6/2/2, v7/1/3
    }
    {
        \coordinate (\name) at (\x,\y);
        \draw[fill] (\name) circle[radius=0.05cm];
    }
\foreach \u/\v in {
     v7/v1, v7/v2, v7/v4, v7/v5, v4/v6, v4/v2, v5/v1, v5/v3,v6/v1, v6/v3, v2/v3}
    {
        \draw (\u) -- (\v);
    }    
\node[anchor=east] at (v1) {$b_1$};
\node[anchor=east] at (v2) {$b_2$};
\node[anchor=east] at (v3) {$b_3$};
\node[anchor=west] at (v4) {$c_3$};
\node[anchor=west] at (v5) {$c_2$};
\node[anchor=west] at (v6) {$c_1$};
\node[anchor=south] at (v7) {$a_3$};
\node[anchor=north] at (1,-0.3) {$I_4$};
\end{tikzpicture}
\hspace{0.5cm}
\begin{tikzpicture}[scale=1]
\foreach \name/\x/\y in {
    v1/0/2, v2/0/1, v3/0/0, v4/2/0, v5/2/1, v6/2/2, v7/1/3
    }
    {
        \coordinate (\name) at (\x,\y);
        \draw[fill] (\name) circle[radius=0.05cm];
    }
\foreach \u/\v in { v7/v2, v7/v3, v7/v5, v7/v6, v6/v1, v6/v3, v5/v1, v5/v2, v4/v1, v4/v2, v4/v3}
    {
        \draw (\u) -- (\v);
    }    
\node[anchor=east] at (v1) {$b_1$};
\node[anchor=east] at (v2) {$b_2$};
\node[anchor=east] at (v3) {$b_3$};
\node[anchor=west] at (v4) {$c_3$};
\node[anchor=west] at (v5) {$c_2$};
\node[anchor=west] at (v6) {$c_1$};
\node[anchor=south] at (v7) {$a_3$};
\node[anchor=north] at (1,-0.3) {$I_5$};
\end{tikzpicture}
\hspace{0.5cm}
\begin{tikzpicture}[scale=1]
\foreach \name/\x/\y in {
    v1/0/2, v2/0/1, v3/0.5/0, v4/1.5/0, v5/2/1, v6/2/2, v7/1/3
    }
    {
        \coordinate (\name) at (\x,\y);
        \draw[fill] (\name) circle[radius=0.05cm];
    }
\foreach \u/\v in { v7/v2, v7/v3, v7/v5, v7/v6, v5/v1, v5/v2, v6/v4, v1/v3, v1/v4, v3/v6, v4/v2}
    {
        \draw (\u) -- (\v);
    }    
\node[anchor=east] at (v1) {$b_1$};
\node[anchor=east] at (v2) {$b_2$};
\node[anchor=east] at (v3) {$b_3$};
\node[anchor=west] at (v4) {$c_3$};
\node[anchor=west] at (v5) {$c_2$};
\node[anchor=west] at (v6) {$c_1$};
\node[anchor=south] at (v7) {$a_3$};
\node[anchor=north] at (1,-0.3) {$I_6$};
\end{tikzpicture}
\begin{tikzpicture}[scale=1]
\foreach \name/\x/\y in {
    v1/0/2, v2/0/1, v3/0.5/0, v4/1.5/0, v5/2/1, v6/2/2, v7/1/3
    }
    {
        \coordinate (\name) at (\x,\y);
        \draw[fill] (\name) circle[radius=0.05cm];
    }
\foreach \u/\v in {v7/v2, v7/v3, v7/v5, v7/v6, v5/v1, v5/v2, v6/v4, v1/v3, v2/v4, v1/v6, v3/v4}
    {
        \draw (\u) -- (\v);
    }    
\node[anchor=east] at (v1) {$b_1$};
\node[anchor=east] at (v2) {$b_2$};
\node[anchor=east] at (v3) {$b_3$};
\node[anchor=west] at (v4) {$c_3$};
\node[anchor=west] at (v5) {$c_2$};
\node[anchor=west] at (v6) {$c_1$};
\node[anchor=south] at (v7) {$a_3$};
\node[anchor=north] at (1,-0.3) {$I_7$};
\end{tikzpicture}
\caption{Induced subgraphs for Subcase 2 in Case 2.}
\label{case 2 subcase 2 subgraph}
\end{figure}
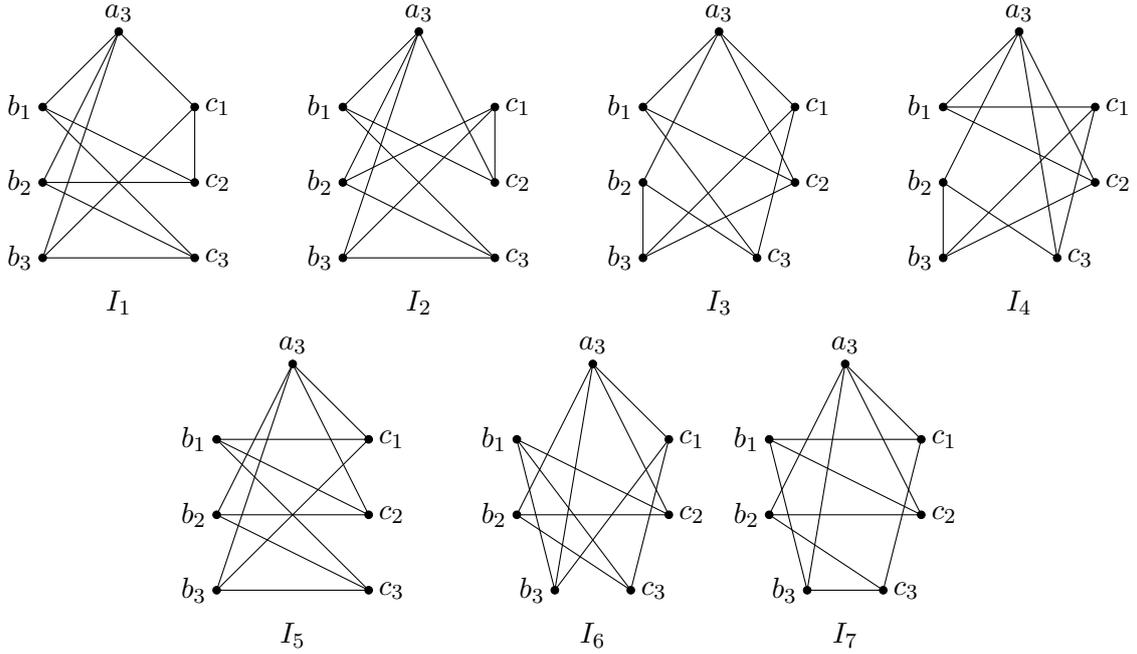
    \item[] Case 3: Suppose $G[N'(i)] \cong  3K_1$ for all $i\in V(H)$. Therefore, the subgraph induced by $V(G)\setminus V(H)$ is a $4$-regular $3$-partite graph with partite sets $N'(i), i\in V(H).$ 
    Subcase 1: Suppose $a_1\in N'(a)$ is adjacent to all the vertices in $N'(b)$ and $c_1$ from $N'(c)$. Since $c_1$ has degree $4$ in $G[V(G)\setminus V(H)]$, it must be adjacent to at least one vertex from $N'(b).$ WLOG, let $c_1\sim b_1$. This gives us another $C_3$ in $G$ with vertices $a_1, b_1, c_1.$ Same as above, the vertices $a_1, b_1, c_1$ must not have any common neighbors. Therefore, $c_1$ must be adjacent to $a_2$ and  $a_3$, and $b_1$ must be adjacent to $c_2$ and $c_3$. Thus, the remaining vertices $\{a_2, a_3, b_2, b_3, c_2, c_3\}$ induces a $3$-regular $3$-partite graph. The only subgraph that meets the given restrictions is shown in  Figure \ref{case 3 subgraph}. Therefore, we can compute the second largest eigenvalue of $G$, which is $\lambda_2(G) = 1$.
\begin{figure}[h]
        \centering
\begin{tikzpicture}[scale=2]

\foreach \name/\x/\y in {
    v1/0/1, v5/1.5/1, v3/0/0, v4/1.5/0, v2/0.5/0.5, v6/1/0.5}
    {
        \coordinate (\name) at (\x,\y);
        \draw[fill] (\name) circle[radius=0.03cm];
    }
\foreach \u/\v in {
    v1/v2, v1/v5, v3/v2, v4/v5, v3/v4, v2/v6, v4/v6, v1/v3, v6/v5}
    {
        \draw (\u) -- (\v);
    }
    
\node[anchor=east] at (v1) {$c_2$};
\node[anchor=north] at (v2) {$b_2$};
\node[anchor=east] at (v3) {$a_2$};
\node[anchor=west] at (v4) {$b_3$};
\node[anchor=west] at (v5) {$a_3$};
\node[anchor=north] at (v6) {$c_3$};

\end{tikzpicture}
\caption{The $3$-regular $3$-partite induced subgraph of $G$ in Case 3.}
\label{case 3 subgraph}
\end{figure}
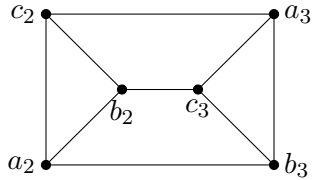
Subcase 2: Suppose $a_1$ is adjacent to two vertices each of the partite sets $N'(b)$ and $N'(c)$. Let $a_1\sim x$, for all $x\in \{b_1, b_2, c_1, c_2\}.$ Suppose $c_1$ is adjacent to all the vertices in $N'(a)$, then we fall into Subcase 1. Suppose $c_1$ is adjacent to all the vertices in $N'(b)$, then $G[a_1,b_1,c_1]\cong C_3$, and $a_1$, $c_1$ have a common neighbor $b_2$. Therefore, $c_1$ can not be adjacent to all the vertices in $N'(a)$ or $N'(b)$; as such, $c_1$ must be adjacent to exactly two vertices in $N'(b).$ Therefore, $c_1$ must be adjacent to at least one of the vertices $b_1, b_2$.  Let $c_1\sim b_1$. This gives us another $C_3$ in $G$ with vertices $a_1, b_1, c_1$. This means $b_1$ must be adjacent to $a_2, c_3$, and $c_1$ must be adjacent to $a_3, b_3$. Note that if $b_2\sim c_2$, then $G[a_1, b_2, c_2]\cong C_3$. The induced subgraph $G[\{a, b_1, c_1\}]$ has an edge, and hence we fall into Case 2. Therefore, $b_2\not\sim c_2$. Since $b_2$ has degree $4$ in $G[V(G)\setminus V(H)]$ and $b_2$ can not be adjacent to a vertex in $\{b_1, b_3, c_1, c_2\}$, we get that $b_2$ must be adjacent to $a_2, a_3$, and $c_3$. Hence, $b_2$ is adjacent to all the vertices in $N'(a)$ and so we fall into Subcase 1. 

\end{itemize}
The remaining case is when $G$ has $n=10$ vertices and girth $3$. We hang the graph by some vertex $v\in V(G)$ and observe that since $G$ is $5$-regular, its diameter equals $2$. There are $59$\footnote{The complete list is available at https://github.com/vishalguptaud/Spectral-Moore-Problem/tree/main.} non-isomorphic graphs of this type \cite{Meringer}. 
We computed the second largest eigenvalue\footnote{The code is available at https://github.com/vishalguptaud/Spectral-Moore-Problem/tree/main.} of each of these $59$ graphs and only one graph, as shown in Figure \ref{fig: n=10}, has the second largest eigenvalue in the interval $(1, \sqrt{5}-1]$ with its second largest eigenvalue being exactly $\sqrt{5}-1$. 

Next, suppose $G$ has girth $4$. For $n=10,$ let $v\in G$ and $N(v)$ be the set of vertices adjacent to $v$ in $G$. Since the girth is 4, $G[N(v)]$ is a coclique. Therefore, the only $5$-regular graph on $10$ vertices that has girth $4$ is $K_{5,5}$ and we know that $\lambda_2(K_{5,5}) = 0.$ For $n\in\{12,14,16\}$, we divide our analysis into the following two cases.

\noindent Case 1: Suppose $G$ is bipartite. For $n=12$, the complement $\overline{G}$ of $G$ is the graph we get after adding a perfect matching between two cliques of order $6$.
$$A(\overline{G}) = \begin{bmatrix}
    J-I&I\\
    I&J-I
\end{bmatrix}.$$
The smallest eigenvalue $\lambda_{min}(\overline{G}) = -2$ with corresponding eigenvector $X= (x,-x)$, where $x$ is a unit vector of length $6$ perpendicular to all-ones vector. Therefore $\lambda_2(G) = -1-\lambda_{min}(\overline{G}) = 1.$ For $n=14$, the complement $\overline{G}$ of $G$ is the graph we get after adding a $2$-regular graph between two cliques of order $7$. The possibilities are $C_4\cup C_{10}, C_4\cup C_4\cup C_6, C_6\cup C_8,$ and $ C_{14}$. Their respective smallest eigenvalues  are $-3, -3, -3, -1-2\cos(\pi/7)\approx -2.8019$. 
Therefore, $\lambda_2(G)= -1-\lambda_{min}(\overline{G}) > \sqrt{5}-1$ in any case. 
For $n=16,$ since $G$ is $5$-regular and bipartite, any two vertices from the same part have a common neighbor and any two vertices from different parts are either adjacent or have a path of length 3 connecting them. Thus, the diameter of $G$ is 3. Suppose $\lambda^2_2(G)< 5$ (if $\lambda^2_2(G)\geq 5$, there is nothing to prove), then by Theorem  \ref{lambda2 bound}, we get $\lambda_2(G)\geq \sqrt\frac{15}{7}>\sqrt{5}-1$. 

\noindent Case 2: Suppose $G$ is not bipartite. Using \cite[Theorem 2]{KS}, we get that the folded $5$-cube is the only non-bipartite $5$-regular triangle-free graph with the second largest eigenvalue less than $\sqrt{2}$. The folded $5$-cube is a strongly regular graph with spectrum $\{5^1, 1^{10}, -3^{5}\}.$ Therefore, there is no $5$-regular non-bipartite graph with girth 4 and the second largest eigenvalue satisfying $1<\lambda_2\leq\sqrt{5}-1.$ This finishes the proof.
\end{proof}

\section{$v(k,\sqrt{2})$}\label{sec:vksqrt2}

In this section, we show that $v(4,\sqrt{2})=14$ and $v(5,\sqrt{2})=16$. 

Consider the following matrix
\begin{equation}
T = \begin{bmatrix}
    0& k &0\\
    1&0& k-1\\
    0&c&k-c
\end{bmatrix}.
\end{equation} 
When $c = (k-2)(\sqrt{2}-1)$, the second largest eigenvalue of $T$, $\lambda_2(T)=\sqrt{2}$.  By the bound \eqref{eq:vklambda}, we get that 
$$v(k,\sqrt{2})\leq M\left(k,3, (k-2)(\sqrt{2}-1)\right)= 1 + k + \frac{k(k-1)}{(k-2)(\sqrt{2}-1)} = \frac{(2+\sqrt{2})k(k-1)-2}{k-2}.$$
Therefore,
\begin{equation}
    v(k,\sqrt{2})\leq \frac{(2+\sqrt{2})k(k-1)-2}{k-2}.
\end{equation}
Denote $\frac{(2+\sqrt{2})k(k-1)-2}{k-2}$ by $N_k$ and we compute the value for small values of $k$. Note that $v(k,\sqrt{2})\leq N_k = \frac{(2+\sqrt{2})k(k-1)-2}{k-2} = (2+\sqrt{2})k + o(k)$.
\begin{table}[h]
\begin{center}
    \begin{tabular}{|c|c|c|c|c|c|c|c|}
    \hline
         $k$ & 4&5&6&7&8&9&10  \\
         \hline
        $\lfloor N_k \rfloor$ &19& 22&25&28&31&34&38 \\
        \hline
    \end{tabular}
    \caption{An upper bound on $v(k,\sqrt{2}).$}
    \label{table 1}
\end{center}
\end{table}

Let $g=5$, then by Lemma \ref{girth restriction}, $n(k,5)\geq \frac{k(k-1)^2-2}{k-2}\geq N_k$ for all $k\geq 3+\sqrt{2}$. Let $g = 6,$ then by Lemma \ref{girth restriction}, $n(k,6)\geq \frac{2(k-1)^3-2}{k-2}\geq N_k$ for all $k\geq 2+\sqrt{2}.$ Since the lower bound in Lemma \ref{girth restriction} increases with $g$, we obtain $n(k,g) > N_k$ for all $k\geq 5$ and $g\geq 5.$ Therefore, for $k\geq 5$, a $k$-regular graph on $n\leq v(k,\sqrt{2})$ vertices has girth at most 4.

Let $G$ be a graph and $H$ be a subgraph of $G$. For $i \geq 0$, we define $\Gamma_i(H)$ as the set of vertices in $G$ at distance exactly $i$ from $V(H)$, and $\Gamma_{\geq i}(H)$ as the set of vertices in $G$ at least at distance $i$ from $V(H)$.

\subsection{$v(4,\sqrt{2})$}
\begin{proposition}
    $v(4,\sqrt{2})=14.$
\end{proposition}
\begin{proof}
The co-Heawood graph is a $4$-regular bipartite graph on 14 vertices with spectrum $\{4^1, \sqrt{2}^6, -\sqrt{2}^6, -4^1\}$. Hence, from Table \ref{table 1}, we get that $14\leq v(4,\sqrt{2})\leq 19$.

Let $G$ be a connected $4$-regular graph on $n=v(4,\sqrt{2})$ vertices such that $\lambda_2(G)\leq \sqrt{2}$. Using Lemma \ref{girth restriction}, we observe that a $4$-regular graph with girth six or more has at least $26$ vertices. Because $n\leq 19$, the girth of $G$ can be at most $5$.

Suppose $G$ has girth 3. Let $H$ be a subgraph of $G$ isomorphic to $C_3$. Because $n\geq 14$, there is a vertex $v$ that is not adjacent to any vertex of $H$. Consider the partition of $G$ with the following three parts : $V(H), \{v\}, V(G) \setminus (V(H)\cup\{v\}).$ The corresponding quotient matrix is
\begin{equation}
Q_n = \begin{bmatrix}
    2&0&2\\
    0&0&4\\
    \frac{6}{n-4}&\frac{4}{n-4}&\frac{4n-26}{n-4}
\end{bmatrix}.
\end{equation}
We compute the second largest eigenvalue $\lambda_2(Q_n) = \frac{n-9+\sqrt{n^2-10n+49}}{n-4}$ and note that $\lambda_2(Q_n)>\sqrt{2}$ because $n\geq 14$. By eigenvalue interlacing, we get that $\lambda_2(G)\geq \lambda_2(Q_n)>\sqrt{2}$, a contradiction. Therefore, there are no connected 4-regular graphs $G$ of order $n\geq 14$ having girth 3 and $\lambda_2(G)\leq \sqrt{2}$.

Suppose the girth of $G$ is 4. Let $H$ be a subgraph of $G$ isomorphic to $C_4$. Because $n\geq 14$, there exists a vertex $v$ that is not adjacent to any vertex of $H$. Consider the partition of $G$ with the following three parts : $V(H), \{v\}, V(G) \setminus (V(H)\cup\{v\})$. The corresponding quotient matrix is
\begin{equation}
Q_n = \begin{bmatrix}
    2&0&2\\
    0&0&4\\
    \frac{8}{n-5}&\frac{4}{n-5}&\frac{4n-32}{n-5}
\end{bmatrix}.
\end{equation}
The second largest eigenvalue $\lambda_2(Q_n) = \frac{n-11+\sqrt{n^2-14n+81}}{n-5}> \sqrt{2}$ if $n\geq 16$. By eigenvalue interlacing, we obtain that $\lambda_2(G)\geq \lambda_2(Q_n)> \sqrt{2}$, a contradiction. Hence, there are no connected 4-regular graphs $G$ of order $n\geq 16$ with girth 4 and $\lambda_2(G)\leq \sqrt{2}$. 

To complete this part of the proof, we need to show that there are no $4$-regular graphs $G$ of girth 4 on 15 vertices with $\lambda_2(G)\leq \sqrt{2}$. By contradiction, assume that $G$ is a $4$-regular graph on $15$ vertices with girth $4$ and $\lambda_2(G)\leq \sqrt{2}$. Let $H$ be a subgraph of $G$ isomorphic to $C_4$. 
Because $|\Gamma_1(H)|\leq 8$, we get that $|\Gamma_{\geq 2}(H)|\geq 15-12=3$.

Let $u,v\in\Gamma_{\geq 2}(H)$. This means that $u$ and $v$ are not adjacent to any vertex of $H$. Consider the following partition of $G$: $V(H), V(G)\setminus (V(H)\cup \{u,v\}), \{u,v\}$. The corresponding quotient matrix is 
\begin{equation}
Q = \begin{bmatrix}
    2&2&0\\
    \frac{8}{9}&4-\alpha-\frac{8}{9} & \alpha\\
    0& 4-\beta & \beta
\end{bmatrix}.
\end{equation}


Depending on whether $u\sim v$ or not, the value of $(\alpha,\beta)$ is $(\frac{2}{3},1)$ or $ (\frac{8}{9},0)$, respectively. The second largest eigenvalue $\lambda_2(Q)$ is $\frac{13+\sqrt{241}}{18}> 1.58$ and $\frac{1+\sqrt{145}}{9}> 1.44$, respectively. By eigenvalue interlacing, we obtain $\lambda_2(G)\geq \lambda_2(Q)>\sqrt{2}$, a contradiction. Hence, there is no connected 4-regular graph $G$ of order $15$ that has girth 4 and $\lambda_2(G)\leq \sqrt{2}$.

Suppose $G$ has girth 5. Let $H$ be a subgraph of $G$ isomorphic to $C_5$.  Since girth is 5, $|\Gamma_1(H)|=10$. Consider the partition of $G$: $V(H), \Gamma_1(H), \Gamma_{\geq 2}(H)$. The corresponding quotient matrix for a given order $n$ is 
$$Q = \begin{bmatrix}
    2&2&0\\
    1&3-\alpha&\alpha\\
    0&4-\beta&\beta
\end{bmatrix}.$$
By Lemma \ref{girth restriction}, a $4$-regular graph of girth 5 has order at least 17. So if $n=17$, then $|\Gamma_{\geq 2}(H)| =2$. Suppose there is an edge in $ \Gamma_{\geq 2}(H)$, then $\alpha = \frac{3}{5}, \,  \beta = 1,   \text{ and } \lambda_2(Q) =\frac{1}{10}(7 + \sqrt{69})>1.53.$ If $n=18$, then $|\Gamma_{\geq 2}(H)| = 3$. Depending on if there are no edges, or one edge, or two edges in  $\Gamma_{\geq 2}(H)$, the value of $(\alpha, \beta)$ is $(\frac{6}{5},0), (1,\frac{2}{3}), (\frac{4}{5},\frac{4}{3}),$ respectively.
We calculate the second largest eigenvalue in each case and observe that
$$\lambda_2(Q)> \begin{cases}
1.45, \text{ if } (\alpha, \beta) = (\frac{6}{5},0),\\
1.53, \text{ if } (\alpha, \beta) = (1, \frac{2}{3}), \text{ and }\\
1.69, \text{ if } (\alpha, \beta) = (\frac{4}{5},\frac{4}{3}).    
\end{cases}
$$
If $n= 19$, then $|\Gamma_{\geq 2}(H)| = 4$. Depending on if there are no edges, or one edge, or two edges, or three edges in  $\Gamma_{\geq 2}(H)$, the value of $(\alpha, \beta)$ is $(\frac{8}{5},0), (\frac{7}{5},\frac{1}{2}), (\frac{6}{5}, 1),$ and $(1,\frac{3}{2}),$ respectively.
We calculate the second largest eigenvalue in each case and observe that
$$\lambda_2(Q)> \begin{cases}
1.51, \text{ if } (\alpha, \beta) = (\frac{8}{5},0),\\
1.56, \text{ if } (\alpha, \beta) = (\frac{7}{5},\frac{1}{2}),\\
1.64, \text{ if } (\alpha, \beta) = (\frac{6}{5},1), \text{ and }\\
1.78
, \text{ if } (\alpha, \beta) = (1, \frac{3}{2}).    
\end{cases}
$$
For the remaining case from above when $n=17$ and $G[\Gamma_{\geq 2}(H)] = 2K_1$, let $\Gamma_{\geq 2}(H) =\{ u,v\}.$ Consider the partition of $G$ in to the following four parts : $V(H), V(G)\setminus (V(H)\cup \{u\}\cup N(u)),$ $N(u), \{u\}.$ Since the girth is 5, $G[N(u)]$ is a coclique on 4 vertices. Thus, the corresponding quotient matrix is 

$$Q = \begin{bmatrix}
    2&\frac{6}{5}&\frac{4}{5}&0\\
    \frac{6}{7}&2&\frac{8}{7}&0\\
    1&2&0&1\\
    0&0&4&0
\end{bmatrix}.$$ The second largest eigenvalue $\lambda_2(Q)> 1.423>\sqrt{2}.$ Therefore by eigenvalue interlacing, we have that $\lambda_2(G)>\sqrt{2}.$ Hence, there is no connected 4-regular graph $G$ of order $n\leq 19$ that has girth 5 and $\lambda_2(G)\leq \sqrt{2}$. This completes the proof.
\end{proof}

\subsection{$v(5,\sqrt{2})$}
\begin{proposition}
 $v(5,\sqrt{2}) = 16.$    
\end{proposition}
\begin{proof}
    By proposition \ref{v(5,sqrt5 -1)}, we know that $v(5,\sqrt{5}-1)=16$. Hence, from Table \ref{table 1} we get that $16\leq v(5,\sqrt{2})\leq 22.$ 
    
Let $G$ be a connected $5$-regular graph on $n=v(5,\sqrt{2})$ vertices such that $\lambda_2(G)\leq \sqrt{2}$. Suppose $G$ has girth three. Let $H$ be a subgraph of $G$ isomorphic to $C_3$. Because $n\geq 16$, there is a vertex $v$ that is not adjacent to any vertex of $H$. Consider the partition of $G$ with the following three parts : $V(H), \{v\}, V(G) \setminus (V(H)\cup\{v\}).$ The corresponding quotient matrix is
    $$Q_n = \begin{bmatrix}
        2&0&3\\
        0&0&5\\
        \frac{9}{n-4}&\frac{5}{n-4}&\frac{5n-34}{n-4}
    \end{bmatrix}.$$
    The second largest eigenvalue $\lambda_2(Q_n) = \frac{n-11 +\sqrt{n^2-12n+81}}{n-4} > \sqrt{2}$ for $n> 13 + 2\sqrt{2} \approx 15.83.$ By eigenvalue interlacing, we obtain that $\lambda_2(G)>\sqrt{2}$ for $n\geq 16$. Therefore, there is no connected $5$-regular graph $G$ of order $n\geq 16$ that has girth 3 and $\lambda_2(G)\leq \sqrt{2}$.
    
   \noindent Next, suppose $G$ has girth four and $n\geq 18$. Let $H$ be a subgraph of $G$ isomorphic to $C_4$. Because $n\geq 18$, we have $|\Gamma_{\geq 2}(H)|\geq 2$. Let $u,v\in \Gamma_{\geq 2}(H).$ Consider the partition of $G$ with the following three parts :  $V(H), \{u,v\}, V(G)\setminus (V(H)\cup\{u,v\}).$ The corresponding quotient matrix is
    $$ Q_n = \begin{bmatrix}2&0&3\\
    0&\alpha&5-\alpha\\
    \frac{12}{n-6}& \frac{2(5-\alpha)}{n-6} & 5 - \frac{22-2\alpha}{n-6}
    \end{bmatrix}.
    $$
    Suppose $\alpha = 1$, i.e., $u\sim v$. We compute the second largest eigenvalue of $Q_n$, $\lambda_2(Q_n) = \frac{3n-38 + \sqrt{n^2-20n+484}}{2(n-6)}$ which is greater than $\sqrt{2}$ for all $n\geq 10.$ By eigenvalue interlacing, we obtain $\lambda_2(G)>\sqrt{2}$. This implies $G[\Gamma_{\geq 2}(H)]$ must be a coclique. So $\alpha  = 0$. The second largest eigenvalue $\lambda_2(Q_n) = \frac{n-17+\sqrt{n^2-14n+169}}{n-6} > \sqrt{2}$ for $n>18+\sqrt{2}\approx 19.41.$ By eigenvalue interlacing, we obtain $\lambda_2(G)>\sqrt{2}$ for $n\geq 20$. Suppose $n=18$ and $|\Gamma_{\geq 2}(H)|\geq 3.$ Let $u,v,w\in \Gamma_{\geq 2}(H).$ Consider the partition of $G$ into the following three parts : $V(H), \{u,v,w\}, V(G)\setminus (V(H)\cup\{u,v,w\}).$ The Corresponding quotient matrix is
    $$Q = \begin{bmatrix}
        2&0&3\\
        0&0&5\\
        \frac{12}{11}&\frac{15}{11}&\frac{28}{11}    
    \end{bmatrix}.$$ 
    The second largest eigenvalue $\lambda_2(Q) = \frac{\sqrt{1345}-5}{22}>1.43$. Hence, by eigenvalue interlacing $\lambda_2(G)>\sqrt{2}$. The remaining case for $n=18$ is when $\Gamma_{\geq 2}(H) = \{u,v\}$ and $u\not\sim v$. 
    Since $G$ is $5$-regular and $|\Gamma_1(H)|=12$, there exists at least two vertices in $\Gamma_1(H)$ (say $w_1, w_2$) that are not adjacent to both $u$ and $v$. Consider the partition of $G$ into the following four parts : $V(H), \{w_1,w_2\}, \{u,v\}, V(G)\setminus (V(H)\cup\{w_1,w_2,u,v\}).$ The corresponding quotient matrix is
    $$Q = \begin{bmatrix}
2&\frac{1}{2}&0&\frac{5}{2}\\
1&\alpha&0&4-\alpha\\
0&0&0&5\\
1&\frac{4-\alpha}{5}&1&\frac{11+\alpha}{5}
    \end{bmatrix}.$$
    Depending on whether $w_1\sim w_2$ or not, we have $\alpha = 1$, $\alpha = 0$, respectively.
    We calculate the second largest eigenvalue in each case and observe that
    $$\lambda_2(Q)>\begin{cases}
        1.47, \text{ if } \alpha =1, \\
        1.419, \text{ otherwise.}
    \end{cases}$$
    By eigenvalue interlacing, we obtain $\lambda_2(G)>\sqrt{2}$. Therefore, there is no connected $5$-regular graph $G$ of order $n\geq 18$ that has girth 4 and $\lambda_2(G)\leq \sqrt{2}$. This completes the proof.
\end{proof}
In subsequent work, the second author \cite{Gupta2026} proved that $v(k,\sqrt{2})=2k+2$ for every $k\geq13$.
\section{Lower bound on second eigenvalue for given girth} \label{sec:7}

In this section, we provide the linear programming bound on the order of a $k$-regular graph for a given girth, which provides an alternative proof of the Moore bound (Lemma \ref{girth restriction}).  
From this result, we present an alternative proof of Theorem~1 (c), (d) in \cite{EKSJ2023}, which concerns upper bounds on algebraic connectivity (the Laplacian second eigenvalue).  
Moreover, the bound can be refined for those pairs $(k,g)$ where the corresponding $(k,g)$-cage is known.

Let $F_i(x)$ be the polynomials defined in \eqref{eq:F012} and \eqref{eq:Fk}.  The following result gives the LP bound for a fixed girth.

\begin{theorem} \label{thm:bound_girth}
Let $G$ be a connected $k$-regular graph of girth $g$ with distinct eigenvalues $k=\lambda_1>\lambda_2>\cdots> \lambda_r$.  
Let $f(x)$ be a polynomial that can be expressed as $f(x)=\sum_{i\ge0}f_iF_i(x)$ with $f_i\in\mathbb{R}$.  
Suppose that  
\begin{enumerate}
\item[$(1)$] $f(k)>0$ and $f(\lambda_i)\ge0$ for any $i \geq 2$;  
\item[$(2)$] $f_0>0$ and $f_i\le0$ for any $i\ge g$.  
\end{enumerate}
Then the order $v$ of $G$ satisfies  
\[
v\ge\frac{f(k)}{f_0}.
\]
\end{theorem}

\begin{proof}
Let $A$ be the adjacency matrix of $G$, with distinct eigenvalues $k=\lambda_1>\lambda_2>\cdots>\lambda_r$.  
Then $A$ admits the spectral decomposition $A=\sum_{i=1}^{r}\lambda_iE_i$.  
Consequently, the matrix $f(A)$ can be written in two equivalent forms:
\[
f(A)=\sum_{i=1}^{r}f(\lambda_i)E_i=\sum_{j\ge0}f_jF_j(A).
\]
Taking traces on both sides gives
\[
f(k)\le\operatorname{tr}\!\left(\sum_{i=1}^{r}f(\lambda_i)E_i\right)
=\operatorname{tr}\!\left(\sum_{j\ge0}f_jF_j(A)\right)
\le vf_0.
\]
In the last inequality, we use the fact (see Theorem~1 in \cite{NozakiLP}) that the $(x,x)$-entry of $F_i(A)$ equals the number of closed non-backtracking walks of length $i$ starting and ending at $x$.  
Hence $\operatorname{tr}F_i(A)=0$ for $0<i<g$, and $\operatorname{tr}(f_iF_i(A))\le0$ for $i\ge g$.  
These inequalities together imply $v\ge f(k)/f_0$.  
\end{proof}
\begin{remark} \label{rem:lp}
    The equality holds in the LP bound if and only if $f(\lambda_i)=0$ for each $i\geq 2$ and $\operatorname{tr}(f_i F_i(A))=0$ for each $i \geq g$. 
    In particular, for the second condition, if $\operatorname{tr}F_i(A)>0$ holds, then $f_i=0$. 
\end{remark}
From the LP bound, we obtain an alternative proof of the Moore bound (Lemma \ref{girth restriction}).

\begin{proof}[Proof of Lemma \ref{girth restriction}]
(1) Apply the LP bound to the polynomial
\begin{equation} \label{eq:lp1}
f(x)=\left(\sum_{i=0}^{d} F_i(x)\right)^2.     
\end{equation}
The polynomial $f(x)$ satisfies the assumptions of the LP bound, and
\[
f(k)=\left(\sum_{i=0}^{d} F_i(k)\right)^2,
\qquad
f_0=\sum_{i=0}^{d} F_i(k),
\]
where $f_0$ is computed from Theorem~3 in \cite{NozakiLP}.
Hence,
\[
n(k,g)\;\ge\;\frac{f(k)}{f_0}
\;=\;\sum_{i=0}^{d} F_i(k)
\;=\;1+k\sum_{i=0}^{d-1}(k-1)^i,
\]
where $F_0(k)=1$ and $F_i(k)=k(k-1)^{\,i-1}$ for $i\ge 1$.

\smallskip
(2) Similarly, consider
\begin{equation}\label{eq:square}
f(x)=(x+k)\left(\sum_{i=0}^{\lfloor (d-1)/2 \rfloor} F_{\,d-1-2i}(x)\right)^2 .
\end{equation}
This polynomial also satisfies the LP conditions, and
\[
f(k)=2k\left(\sum_{i=0}^{\lfloor (d-1)/2 \rfloor} F_{\,d-1-2i}(k)\right)^2,
\qquad
f_0=k\sum_{i=0}^{\lfloor (d-1)/2 \rfloor} F_{\,d-1-2i}(k).
\]
Therefore,
\[
n(k,g)\;\ge\;\frac{f(k)}{f_0}
\;=\;2\sum_{i=0}^{\lfloor (d-1)/2 \rfloor} F_{\,d-1-2i}(k)
\;=\;2\sum_{i=0}^{d-1}(k-1)^i,
\]
where $F_i(k)=k(k-1)^{\,i-1}=(k-1)^{\,i-1}+(k-1)^i$ for $i\ge 1$.
\end{proof}

Theorem~\ref{thm:bound_girth} is an analogue of the linear programming bound for spherical $t$-designs \cite[Theorem~5.10]{DGS77}.  
The LP bound for spherical $t$-designs can occasionally improve upon the absolute bound, although only a single such case is known so far (see \cite{BD2001}).  
It remains an interesting question whether there exists a pair $(k,g)$ for which the LP bound for girth improves the Moore bound.

The following result was proved in \cite[Theorem~6.3 and Remark~6.4]{CKMNO2022}.

\begin{theorem}\label{thm:bound}
Let $G_j(x)=\sum_{i=0}^{j}F_i(x)$, and let $\tau_j$ denote the largest zero of $G_j(x)$.  
For $\theta\in (-1,2\sqrt{k-1})$, there exists an integer $d\ge1$ such that $\tau_{d-1}<\theta\le\tau_d$.  
Then  
\[
v(k,\theta)\le M(k,\theta)
:=1+k\sum_{i=0}^{d-2}(k-1)^i+\frac{k(k-1)^{d-1}}{c_\theta},
\]
where $c_\theta=-F_d(\theta)/G_{d-1}(\theta)$. Also, $M(k,\theta)$ is monotonically increasing for $\theta\in(-1,2\sqrt{k-1})$.  
\end{theorem}

We now give an alternative proof of the following theorem from \cite{EKSJ2023}.

\begin{theorem}[Exoo, Kolokolnikov, Janssen and Salamon \cite{EKSJ2023}]
Let $\lambda$ denote the second eigenvalue of a $k$-regular Moore graph with girth~$g$, i.e., a graph attaining the Moore bound.  
Then the second eigenvalue of any connected $k$-regular graph with girth~$g$ is greater than or equal to~$\lambda$.  
\end{theorem}
\begin{proof}
We use the notation from Theorem~\ref{thm:bound}.  
From Remark \ref{rem:lp} and \eqref{eq:lp1}, the non-trivial eigenvalues of the Moore graph of girth $2d+1$ are the zeros of $\sum_{i=0}^d F_i(x)=G_d(x)$. 
The second eigenvalue of the Moore graph is $\tau_d$ and 
\[c_{\tau_d}=-\frac{F_d(\tau_d)}{G_{d-1}(\tau_d)}
=-\frac{G_d(\tau_d)-G_{d-1}(\tau_d)}{G_{d-1}(\tau_d)}=1. 
\] 
The Moore graph attains the bound $v(k,\tau_d)\le M(k,\tau_d)$.  
If the second eigenvalue of another graph is smaller than~$\tau_d$, then by Theorem~\ref{thm:bound} its order satisfies $v(k,\lambda_2(G))<M(k,\tau_d)$, contradicting the fact that a graph of girth~$2d+1$ must have order at least $M(k,\tau_d)$.

From Remark \ref{rem:lp} and \eqref{eq:square}, 
the second eigenvalue of the Moore graph of girth $2d$ 
is the largest zero $\kappa_d$ of 
\[(x+k)\left(\sum_{i=0}^{\lfloor (d-1)/2 \rfloor} F_{\,d-1-2i}(x)\right)=
kG_{d-1}(x)+F_d(x).\] 
Then, we have 
\[
c_{\kappa_d}=-\frac{F_d(\kappa_d)}{G_{d-1}(\kappa_d)}=-\frac{F_d(\kappa_d)}{-F_d(\kappa_d)/k}=k. 
\]
The Moore graph again attains the bound 
$v(k,\kappa_d)\le M(k,\kappa_d)$, and the same argument as above applies.  
\end{proof}

A $k$-regular graph with girth~$g$ is called a \emph{$(k,g)$-cage} if its order attains the minimum possible value~$n(k,g)$. From \cite{EJ2008}, we know that apart from the Moore graphs, the known $(k,g)$-cages are as follows:
\begin{align*} &n(3,7)=24, & &n(3,9)=58,& &n(3,10)=70, & &n(3,11)=112,& n(4,5)=19, \\ &n(4,7)=67, & &n(5,5)=30, & &n(6,5)=40, & &n(7,6)=90. \end{align*}

By Theorem~\ref{thm:bound}, there exists a value $\theta\in (-1,2\sqrt{k-1})$ such that $M(k,\theta)=n(k,g)$.  
If the second eigenvalue of a $k$-regular graph is smaller than~$\theta$, then its order is smaller than~$n(k,g)$, which is impossible.  
Hence any $k$-regular graph with girth~$g$ must have second eigenvalue at least~$\theta$.

This observation refines the upper bound $AC(k,g)$ on the Laplacian second eigenvalue $k-\lambda_2$ given in~\cite[Table~1]{EKSJ2023}, for the cases where cages are known:
{\small \begin{align*} &AC(3,7)=1.133457	 \, (1.1864), & &AC(3,9)=0.732465\,  (0.8088),& &AC(3,10)=0.676596\, (0.7118), &\\
&AC(3,11)=0.572485\, (0.6069),&
&AC(4,5)=2.59146\, (2.6972), & &AC(4,7)=1.63449\, (1.7466), & \\
&AC(5,5)=3.31619\, (3.4384), & &AC(6,5)=4.14832\, (4.2087), &  &AC(7,6)=4.51037\, (4.5505). & \end{align*}}
Here, the values in the parentheses are the known upper bounds given in \cite{EKSJ2023}. 
Note that Table~1 in~\cite{EKSJ2023} omits boldface for the entries corresponding to the Moore graphs of $(k,g)=(7,5), (i,8), (i,12)$ with $i=4,5,6,8,9,10$.  

\section{Alon-Boppana-type bounds} \label{sec:6}

In this section, we show how \eqref{eq:vklambda} can be used to obtain
various Alon--Boppana-type bounds.


The Alon-Boppana theorem is a fundamental result in spectral graph theory that sets the benchmark for the second eigenvalues of sequences of
$k$-regular Ramanujan graphs. In \cite{A1986}, Alon stated the Alon-Boppana theorem as follows: for any $k$-regular graph $G$ with $n$ vertices, $\lambda_2(G)\geq 2\sqrt{k-1}-\BigOh(\log_{k-1}n)^{-1}$. Later on, Nilli \cite{Nilli1} proved that if $G$ is a connected $k$-regular graph that contains two edges at distance at least $2\ell+2$, then 
$\lambda_2(G)\geq 2\sqrt{k-1}-\frac{2\sqrt{k-1}-1}{\ell+1}$. A better bound was obtained by Friedman
(see Prop. 3.2 and Cor. 3.6 in \cite{Fri1}), who showed that
if $G$ is a $k$-regular graph with diameter $D\geq 2r$,
then
$\lambda_2(G)\geq 2\sqrt{k-1}\cos\left(\frac{\pi}{r+1}\right)$.
This is the same as stating that if $G$ is a $k$-regular graph
with $\lambda_2(G)<2\sqrt{k-1}\cos\left(\frac{\pi}{r+1}\right)$,
then $D\leq 2r-1$.
The Moore or degree-diameter bound implies that
\begin{equation}\label{eq:ab1}
    v\leq 1+k+k(k-1)+\ldots+k(k-1)^{2r-2}.
\end{equation}

Note that if $\theta=2\sqrt{k-1}\cos\left(\frac{\pi}{r+1}\right)$, then \eqref{eq:vklambda} implies that any connected $k$-regular graph $G$ with $\lambda_2(G)\leq\theta$ and $v$ vertices must satisfy
\begin{equation}\label{eq:ab2}
v<1+k+k(k-1)+\ldots +k(k-1)^r.
\end{equation}
See \cite{CioabaRIMS} for a proof or the next theorem for a stronger result. 

Motivated by the Alon-Boppana theorem and based on computational results, Kolokolnikov \cite[Conj. 4.5]{Kolo15} conjectured that if a cubic graph has order $2^{d+1}-2$, then its algebraic connectivity (smallest positive eigenvalue of the Laplacian) is at most $3-2\sqrt{2} \cos (\pi/d)$ or equivalently, its second eigenvalue is at least $2\sqrt{2}\cos (\pi/d)$. 

We now prove the following general result for any valency $k\geq 3$ (the case $k=3$ corresponding to Kolokolnikov's conjecture).
\begin{theorem} 
Let $k\geq 3$ and $d\geq 2$ be integers. 
Any connected $k$-regular graph of order at least $(2(k-1)^{d}-2)/(k-2)$ has algebraic connectivity at most $k-2\sqrt{k-1}\cos(\pi/d)$ or second adjacency matrix eigenvalue at least $2\sqrt{k-1}\cos(\pi/d)$. 
\end{theorem}
\begin{proof}
We show that if the second eigenvalue is smaller than $2\sqrt{k-1}\cos(\pi/d)$, then the order is smaller than $(2(k-1)^{d}-2)/(k-2)$. From Theorem~\ref{thm:bound}, it suffices to prove that
\begin{equation*}
M(k,2\sqrt{k-1}\cos(\pi/d))=\frac{2(k-1)^{d}-2}{k-2}.
\end{equation*}
    
From \cite[III.3, equation (3.9)]{BannaiIto84}, we know that $\tau_{d-1}<2\sqrt{k-1}\cos(\pi/d)<\tau_d$ and  
\begin{align*}
G_d(2\sqrt{k-1}\cos t)&=\frac{(k-1)^{\frac{d-1}{2}}}{\sin t}(\sqrt{k-1} \sin(d+1)t+\sin dt), \\
G_{d-1}(2\sqrt{k-1}\cos t)&=\frac{(k-1)^{\frac{d-2}{2}}}{\sin t}(\sqrt{k-1} \sin dt+\sin(d-1)t).
\end{align*}

Therefore, 
\begin{align*}
c_\theta&=-\frac{F_d(2\sqrt{k-1}\cos(\pi/d))}{G_{d-1}(2\sqrt{k-1}\cos(\pi/d))}=1-\frac{G_d(2\sqrt{k-1}\cos(\pi/d))}{G_{d-1}(2\sqrt{k-1}\cos(\pi/d))}=k. 
    \end{align*}
Finally, we get that
\[
M(k,2\sqrt{k-1}\cos(\pi/d))=1+k\sum_{i=0}^{d-2}(k-1)^i+\frac{k(k-1)^{d-1}}{k}=\frac{2(k-1)^{d}-2}{k-2}. \qedhere 
\]
\end{proof}

For clarity, note that the polynomials $(F_i(x))_{i\geq 0}$ involved in the equation (3.9) in \cite[III.3]{BannaiIto84} are our polynomials $(G_i(x))_{i\geq 0}$ from this proof. Indeed, the sequence $(G_i(x))_{i \geq 0}$ satisfies $G_0(x)=1$, $G_1(x)=x+1$, $G_2(x)=x^2+x-(k-1)$ and, for $i\geq 3$, 
\begin{align*}
G_i(x)
&=\sum_{j=0}^i F_j(x)=\sum_{j=3}^i \bigl(xF_{j-1}(x)-(k-1)F_{j-2}(x)\bigr)+(x^2-k)+x+1\\
&=\sum_{j=3}^i \bigl(xF_{j-1}(x)-(k-1)F_{j-2}(x)\bigr)
   +x\bigl(F_1(x)+F_0(x)\bigr)-(k-1)F_0(x)\\
&=xG_{i-1}(x)-(k-1)G_{i-2}(x).
\end{align*}
Thus, the sequence of polynomials $(G_i(x))_{i\geq 0}$ satisfies the same three-term recurrence
\[
G_i(x)=xG_{i-1}(x)-(k-1)G_{i-2}(x)\quad (i\geq 3)
\]
as the sequence $(F_i(x))_{i\geq 0}$ appearing in equation~(3.9) in \cite[III.3]{BannaiIto84};
see also equation~(3.3) in \cite[III.3]{BannaiIto84} for the definition of $(F_i(x))_{i\geq 0}$.

As seen from Friedman's bound, the diameter of a regular
graph provides a lower bound on its second eigenvalue.  It is of interest to obtain lower bounds for the second eigenvalue of regular graphs of small diameter or of large valency and some recent work has been done in this direction \cite{Balla,Ihringer,RST,Zhang}. We now show how one can use \eqref{eq:vklambda} to obtain such bounds. 
\begin{proposition}\label{prop:Tk3c}
Let $k\geq 3$ be an integer. 
\begin{enumerate}
    \item Let $\theta\in (0,\sqrt{k})$. If $G$ is a $k$-regular graph with $v>1+k+\frac{k(k-1)(\theta+1)}{k-\theta^2}$ vertices, then $\lambda_2(G)>\theta$.
    \item Let $\alpha\in (0,1)$. If $G$ is a $k$-regular graph with $v>1+k+\frac{(k-1)(\sqrt{k}+1)}{1-\alpha^2}$ vertices, then $\lambda_2(G)>\alpha \sqrt{k}$.
\end{enumerate}
\end{proposition}
\begin{proof}
\begin{enumerate}
\item Let $\theta\in (0,\sqrt{k})$. Take $c=\frac{k-\theta^2}{\theta+1}$ and consider the matrix
\begin{equation*}
T=T(k,3,c)= \begin{bmatrix}
    0&k&0\\
    1&0&k-1\\
    0&c&k-c
\end{bmatrix}.
\end{equation*}
Its characteristic polynomial is 
\begin{equation*}
(x-k)[F_2(x)+c(F_1(x)+F_0(x))]=(x-k)(x^2+cx+c-k).
\end{equation*}
From our choice of $c$, we get that $\lambda_2(T)=\theta$. By the bound \eqref{eq:vklambda}, we get that 
\begin{equation*}
    v(k,\theta)\leq 1+k+\frac{k(k-1)}{c}=1+k+\frac{k(k-1)(\theta+1)}{k-\theta^2}.
\end{equation*}

\item 
This follows from the first part by taking $\theta=\alpha\sqrt{k}$.
\end{enumerate}\end{proof}

When $\alpha=1/2$, we obtain that if $v>1+k+\frac{2(k-1)(\sqrt{k}+2)}{3}$, then $\lambda_2(G)\geq \frac{\sqrt{k}}{2}$. This proves the first part of \cite[Theorem 1.1]{Balla} or \cite[Theorem 1.6]{RST}. Using Proposition \ref{prop:Tk3c} and some technical calculations which we omit here, we can also show that if $k$ is sufficiently large and $G$ is a $k$-regular graph with $v$ vertices and $v^{2/3}<k\leq v^{3/4}$, then $\lambda_2(G)>\frac{n}{2k}-1$. This corresponds to the second part of the theorems mentioned above.

\end{document}